\documentclass{conm-p-l}

\usepackage{amssymb}

\usepackage{graphicx}

\newtheorem{prop}{Proposition}
\newtheorem{lem}{Lemma}
\newtheorem{theorem}{Theorem}

\theoremstyle{definition}
\newtheorem{za}{Remark}
\newtheorem{opr}{Definition}

\newcommand{\Cb}{{\bf C}}
\newcommand{\Z}{\mathbb{Z}}
\newcommand{\R}{\mathbb{R}}
\newcommand{\I}{\rm{I}}
\newcommand{\II}{\rm{II}}
\newcommand{\Sk}{\operatorname{Sk}}
\newcommand{\Skdir}{\Sk_{\fam0 dir}}
\newcommand{\Skud}{\Sk_{\fam0 ud}}
\newcommand{\base}{D}
\newcommand{\bbase}{\bar\base}
\newcommand{\Gf}{\varphi}
\newcommand{\XR}{X_{\R}}

\newcommand{\inserthyphen}{\ifcat\next a-\fi\ignorespaces}
\newcommand{\pblack}{$\bullet$\futurelet\next\inserthyphen}
\newcommand{\pwhite}{$\circ$\futurelet\next\inserthyphen}
\newcommand{\pcross}{$\vcenter{\hbox{$\scriptstyle\times$}}$\futurelet\next\inserthyphen}
\newcommand{\black}{\protect\pblack}
\newcommand{\white}{\protect\pwhite}
\newcommand{\cross}{\protect\pcross}

\begin{document}

\title[Maximally inflected real trigonal curves]{Maximally inflected real trigonal curves on Hirzebruch surfaces}

\author{V.I. Zvonilov.}
\address{Institute of Information Technology, Mathematics and Mechanics, Lobachevsky State University of Nizhny Novgorod, Nizhny Novgorod, Russia, 603950 }

\email{zvonilov@gmail.com}
\thanks{The author's work was done on a subject of the State assignment (N 0729-2020-0055).}

\begin{abstract}
		In 2014 A. Degtyarev, I. Itenberg and the author gave a description, up to fiberwise equivariant deformations, of maximally inflected real
	trigonal curves of
	type~I (over a base $ B $ of an arbitrary genus) in terms of the combinatorics of sufficiently simple
	graphs   
	and for $ B=\mathbb{P}^1 $ obtained a complete classification of such curves.	
	In this paper, the mentioned results  are extended to maximally inflected real
	trigonal curves of
	type~II over $ B=\mathbb{P}^1 $. 	
		
\end{abstract}

\maketitle

\section{Introduction}
Let $\pi: \Sigma \rightarrow B$ be a geometrically ruled surface over a
base~$B$
and with the real exceptional section~$E$, $E^2=~-d<0$.
A \emph{real trigonal
	curve}  is a reduced real curve $C\subset\Sigma$ disjoint from $E$ and such that the restriction $\pi_C:C\rightarrow B$ of~$\pi$ has degree three\footnote{Thus, throughout the paper a trigonal curve is not abstract but embedded into a ruled surface.}.
Such a curve $ C $ is \emph{maximally inflected} if it is nonsingular and
all critical points of $ \pi_C $ are simple and real.
The notion of maximally inflected real trigonal curve
can be considered as some kind of a generalization of a real trigonal curve that is maximal in the sense of the
Smith--Thom
inequality (for details see Introduction in~\cite{DIZ}).

 A nonsingular real algebraic curve~$C$  is of \emph{type~$\I $} if its real part (i.e. the set of its real points) is nonempty and the fundamental class of the real part vanishes in $H_1(C; \Z_2)  $; otherwise, $C$ is of \emph{type~$\II$}.
For maximally inflected curves of type I 
over a base of any genus, a description up to equivariant deformations was done in~\cite{DIZ}. The present paper develops the techniques used in~\cite{DIZ} and extends its results to maximally inflected
trigonal curves of any
type over the rational base (\emph{i.e.} over $ B=\mathbb{P}^1 $). 

Any Jacobian
elliptic surface is obtained from a ruled surface $ \Sigma $ and a trigonal curve $C\subset\Sigma$ by resolving singularities of the double covering of $\Sigma$ ramified at $ C\cup E $ where  $E\subset\Sigma$ is the exceptional section.    Thus, as in~\cite{DIK} and \cite{DIZ}, the results for maximally inflected curves can be extended, almost literally, to
real Jacobian
elliptic surfaces  with all singular fibers real.

Throughout the paper, all varieties are over~$\Cb $
(usually, with a real structure)
and nonsingular.

\subsection{Principal results}
As in~\cite{DIZ}, 
the principal tool used in the paper
is the notion of \emph{dessin}, see
Section~\ref{rtc},
which is a real
version of Grothendieck's \emph{dessin d'enfants} of the
functional
$j$-invariant of the curve;
this concept was originally suggested by
S.~Orevkov~\cite{O} and then developed in~\cite{DIK} and in the  book \cite{Degt}, where
the study of deformation classes of real trigonal curves was
reduced to that of dessins, see Proposition~\ref{equiv.curves}.
In the
case of maximally inflected curves over the rational base, we manage to
simplify
the 
rather complicated combinatorial 
structure 
of dessins 
to somewhat
smaller
graphs, which are called \emph{skeletons}, see
Section~\ref{S.skeletons}. 
The principal result of the paper is
Theorem~\ref{cor.Sk} which
establishes a one-to-one correspondence between the equivariant
fiberwise deformation classes of maximally inflected trigonal
curves over the rational base and certain equivalence classes of skeletons.

In the paper, the study is restricted to the case of rational base because it is not difficult to construct an example of two nonequivalent, up to fiberwise equivariant deformations, maximally inflected
	trigonal curves of type~II over a base of positive genus with the same skeleton. Thus, below all trigonal curves are over the rational base.
	
 Theorem~\ref{cor.Sk} is used to
 a constructive description
of the real parts realized by maximally inflected 
trigonal curves, see 
Section~\ref{S.rational}.

As already noted in~\cite{DIZ},  
along with the deformation equivalence, there is a weaker relation, the so called
\emph{rigid isotopy}, which does not take the ruling of the surface containing trigonal curves into account and is defined as a path in the space of nonsingular trigonal curves on the surface. It turns out that, up to rigid isotopy, the assumption that a curve over the rational base should be maximally
inflected is not very restrictive (see Theorem \ref{max.inflected}).  In the last part of the paper it is proved that the weak equivalence differs from the deformation one even for curves of type~I.

\subsection{Contents of the paper}
Sections~\ref{S.curves} and~\ref{rtc}, following \cite{DIZ} and \cite{DIK},  recall a few notions and facts
related to topology of real trigonal curves and their dessins,
respectively.
In Section~\ref{S.skeletons}, we introduce
skeletons, define their equivalence, and prove
Theorem~\ref{cor.Sk}. In Section~\ref{S.rational} we introduce \emph{blocks}, which are the
`elementary pieces' constituting the dessin of
any maximally inflected curve.
Finally, in Section~\ref{s.rigid} we compare the weak equivalence with the deformation one.

\section{Trigonal curves and elliptic surfaces}\label{S.curves}

This section, following \cite{DIZ},  recalls a few basic notions and facts
related to topology of real trigonal curves
and elliptic surfaces.

A \emph{real variety}
is a complex algebraic variety~$V$  with
an anti-holomorphic involution $c = c_V: V \rightarrow V$. 
The fixed point set $V_{\R} = \mathrm{Fix}\, c$ is called the \emph{real part}
of~$V$. A regular morphism $f:V \rightarrow W$ of two real varieties
is called \emph{real}, or \emph{equivariant}
if $f\circ c_V=c_W\circ f$.

\subsection{Real trigonal curves}\label{real-trigonal}
Let $\pi:\Sigma_d\rightarrow \mathbb{P}^1$ be a real Hirzebruch surface (rational geometrically ruled surface) with the real exceptional section~$E$, $E^2=~-d<0$, and $C\subset\Sigma_d$ be a real trigonal curve. 
The fibers of the ruling~$\pi$ are  called
\emph{vertical}, \emph{e.g.}, we speak about vertical tangents, vertical
flexes \emph{etc}. 
In an affine chart  on $\Sigma_d$, the exceptional section $ E $ and the curve   $ C$ are defined by equations $ y=\infty $ and
$ y^3+b(x)y+w(x)=0, $
where $b$ and $w$ are certain sections that can be considered as homogeneous polynomials of degrees $ 2d $ and $ 3d $. 
 
A nonsingular real trigonal curve   $ C$ is \emph{almost generic} 
 if $ C $ is nonsingular and all zeros of the discriminant
$\Delta(x)=4b^3+27w^2$ are simple.  It is 
\emph{hyperbolic} 
if the restriction $C_{\R} \rightarrow \mathbb{P}^1_{\R}$ of~$\pi$ is three-to-one.

For a real trigonal curve $ C:  y^3+b(x)y+w(x)=0$, the rational function $j=j_C=\frac{4b^3}{\Delta}=1-\frac{27w^2}{\Delta}$ is the \emph{$j$-invariant} of $ C $.
An almost generic curve is \emph{generic} if, for each real critical value $ t $ of its  $j$-invariant, 
the
ramification index of each pull-back of~$t=0$ (respectively,~$t\neq 0$)
equals~$3$ (respectively,~$2$) and all pull-backs of $ t\notin \{0,1\}$ are real. Any almost generic
real
trigonal curve can be perturbed to a generic one.

The real part of a non-hyperbolic curve~$C$ has a unique
\emph{long component} $l$, characterized by the fact
that the restriction $l \rightarrow \mathbb{P}^1_{\R}$ of~$\pi$
is of degree~$\pm 1$. All other components of~$C_{\R}$
are called \emph{ovals}; they are mapped to~$\mathbb{P}^1_{\R}$
with degree~$0$. Let $Z \subset \mathbb{P}^1_{\R}$ be the set of points
with more than one preimage in $C_{\R}$. Then, each oval
projects to a whole component of~$Z$, which is also called
an \emph{oval}. The other components of $Z$,
as well as
their preimages
in~$l$, are called \emph{zigzags}.

\subsection{Deformations}\label{deformation}
Throughout this paper, by a \emph{deformation} of a trigonal curve
$C\subset\Sigma_d$  we mean a deformation of
the pair $(\pi:\Sigma_d\rightarrow \mathbb{P}^1, C)$
in the sense of Kodaira--Spencer.
A deformation of an almost generic curve
is called \emph{fiberwise} if the curve remains
almost generic throughout the deformation.

\emph{Deformation equivalence} of real trigonal curves is
the equivalence relation generated by equivariant
fiberwise deformations and
real isomorphisms.

\subsection{Jacobian surfaces}\label{s.If}
A real surface~$X$ is said to be of \emph{type~$\I$} if
$[\XR]=w_2(X)$ in $H_2(X;\Z_2)$.

Let~$C \subset \Sigma_d$ be a nonsingular trigonal curve.
Assume that $d=2k$
is even and
consider a double covering $p:X\rightarrow\Sigma_d$ of~$\Sigma_d$ ramified
at $C+E$. It is a Jacobian elliptic surface.

The following statement is a consequence of \cite[2.4.2, 2.5.2]{DIZ}.
\begin{prop} \label{IF=IB}
	Let $C\subset\Sigma_{2k}$ be a real trigonal curve. Then, a
	Jacobian elliptic surface~$X$ ramified at $C+E$ is of
	type~$\I$ if and only if $C$ is of type~$\I$.
\end{prop} 

\section{Dessins}\label{rtc}
Instead of usual approach to the notion of dessin as a special case  of a trichotomic graph (see~\cite[5]{DIK} and ~\cite[3]{DIZ}), in this paper, a \emph{dessin}   is defined  as a graph $ \Gamma$ in a disk isomorphic to the dessin of a generic real trigonal curve (see below). Such a definition is relevant due to Proposition \ref{equiv.curves} below.

Let $ D $ be the quotient of $ \mathbb{P}^1 $ by the complex conjugation and $\mathrm{pr}: \mathbb{P}^1 \rightarrow D$ be the projection.  For points, segments, \emph{etc.}
situated at the boundary~$\partial\base$, the term \emph{real} is used.
For the $ j $-invariant $j_C: \mathbb{P}^1\rightarrow \mathbb{P}^1= \Cb\cup\{\infty\}$ of a generic real trigonal curve $ C\subset\Sigma_d $,  decorate $ \mathbb{P}^1_{\R}$ lying in  the target Riemann sphere and  endowed with the positive
orientation (\emph{i.e.}, the order of~$\R$), as follows: $0$, $1$,
and~$\infty$ are, respectively, \black-, \white-, and
\cross-vertices; $(\infty,0)$, $(0,1)$, and $(1,\infty)$ are, respectively, \rm{solid}, \rm{bold}, and \rm{dotted} edges.   
The accordingly oriented and decorated  graph $\Gamma_C=\mathrm{pr}(j_C^{-1}(\mathbb{P}^1_{\R}))$ is a \emph{dessin}  of $ C $. Its \black-, \white-, and
\cross-vertices, which are the branch points with   critical values $ 0,1, \infty $, are called \emph{essential}, the other ones,  which are the
branch points with  real critical values other than $ 0, 1, \infty $, are called \emph{monochrome} vertices.  Since $ C $ is generic, $ \Gamma_C $ has no \emph{inner} (\emph{i.e.}, not real) monochrome vertices and the valency of any monochrome vertex is $ 3 $.

The monochrome vertices can
further be subdivided into solid, bold, and dotted,
according to their
incident edges.
A \emph{monochrome cycle} in~$\Gamma_C$ is a cycle with all vertices
monochrome, hence all edges and vertices of the same kind.
The definition of dessin implies
that~$\Gamma_C$ has no oriented monochrome cycles. 

The \emph{degree} of $ \Gamma_C $ is~$ 3d$.  A dessin of degree~$3$  is
called a \emph{cubic}.

In the drawings, (portions of) the \emph{real part} $ \partial D \cap\Gamma $ of a dessin $ \Gamma $
are indicated by wide grey lines.

For a dessin
$\Gamma\subset\base$, we denote by~$D_{\Gamma}$ the closed cut
of~$D$ along~$\Gamma$.
The connected components of~$D_{\Gamma}$
are called \emph{regions} of~$\Gamma$. A region with three essential vertices in
the boundary is called \emph{triangular}.

\subsection{Pillars.}
\label{pillar}
A dessin $\Gamma$
 is
called
 \emph{hyperbolic}, if all its real
	edges  are dotted. It corresponds to a hyperbolic curve.

For some real edges  of the same kind of the decoration of a dessin $ \Gamma $, if a union of the closures of them
  is homeomorphic to
a closed interval,
this union is called a \emph{segment}.
A \rm{dotted} (\rm{bold})
segment
is called \emph{maximal}
if it is bounded by
two
\cross-
(respectively, \black-) vertices.
A \emph{dotted}/\emph{bold} \emph{pillar} ($ n $-\emph{pillar}) is
 a maximal \rm{dotted}/\rm{bold} segment (with the number of \white-vertices being equal $ n $).
 
 If $\Gamma$ 
 is non-hyperbolic,
 the ovals and zigzags in $ \mathbb{P}^1_{\R} $ are represented
 by the  \rm{dotted} pillars
 of~$\Gamma$, which contain even and odd number of \white-vertices, respectively.
  
 A bold pillar
 with an even/odd number of
 \white-vertices  
 is called a
 \emph{wave/jump}. 

\subsection{ }\label{ss.equivalence}
Two dessins are called
\emph{equivalent} if, after a homeomorphism
of the disk, they
are
connected by a finite sequence of
isotopies and the following \emph{elementary moves}:
\begin{itemize}
	\item[--]
	\emph{monochrome modification}, see
	Figure~\ref{fig.moves}(a);
	\item[--]
	\emph{creating} (\emph{destroying}) a bridge\emph{}, see
	Figure~\ref{fig.moves}(b),
	where a \emph{bridge} is a pair of
	monochrome vertices connected by a real monochrome edge;
	\item[--]
	\emph{\white-in} and its inverse \emph{\white-out}, see
	Figure~\ref{fig.moves}(c) and~(d);
	\item[--]
	\emph{\black-in} and its inverse \emph{\black-out}, see
	Figure~\ref{fig.moves}(e) and~(f).
\end{itemize}
(In the first two cases, a move is
considered valid
only if the result
is again a dessin.
In other words,
one needs to check
the absence of oriented monochrome cycles.)
An equivalence of two dessins 
is called \emph{restricted} if the homeomorphism
is identical and the isotopies above can be chosen
to preserve the pillars (as sets).

\begin{figure}[tb]
	\begin{center}
		\includegraphics{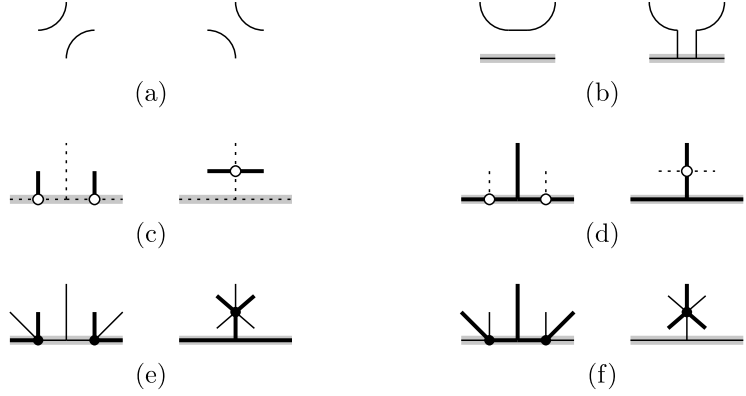}\\
	\end{center}
	\caption{Elementary moves of dessins}\label{fig.moves}
\end{figure}

The following statement is proved in~\cite[5.3]{DIK}.

\begin{prop} \label{equiv.curves}
	Each dessin~$\Gamma$ is of the form~$\Gamma_C$
	for some generic real trigonal curve~$C$.
	Two generic real trigonal curves are deformation
	equivalent
	\emph{(}in the
	class of
	almost generic
	real trigonal curves\emph{)}
	if and only if their dessins are equivalent.
	\qed
\end{prop} 

\subsection{Indecomposable dessins}\label{normal-forms}

For $ j=1,2 $, let $ D_j $ be a disk, $\Gamma_j\subset D_j$ a dessin,  $I_j\subset\partial D_j$
a segment whose endpoints are
not vertices of~$\Gamma$. Let
$\Gf:I_1\to I_2$ be an isomorphism, \emph{i.e.}, a diffeomorphism of the
segments establishing a graph isomorphism
$\Gamma\cap I_1\to\Gamma\cap I_2$. 
Consider the
quotient $D_\Gf=(D_1\sqcup D_2)/\{x\sim\Gf(x)\}$ and the image
$\Gamma'_\Gf\subset D_\Gf$ of~$\Gamma$. Denote by~$\Gamma_\Gf$
the graph obtained from~$\Gamma'_\Gf$ by erasing the image
of~$I_1$, if $\Gf$ is orientation reversing, or converting the
images of the endpoints of~$I_1$ to monochrome vertices otherwise.

In what follows we always
assume that either $I_j$ is part of an edge of~$\Gamma_j$ or
$I_j$ contains a single \white-vertex. In the latter case, $\Gf$
is unique up to isotopy; in the former case, $\Gf$ is determined
by whether it is orientation preserving or orientation reversing.
If $\Gamma_\Gf$ is a dessin, it is called the result of
\emph{gluing} $\Gamma_1$, $\Gamma_2$ along~$\Gf$.
The image of~$I_1$ is called a \emph{cut}
in~$\Gamma_\Gf$. 
The cut is
called \emph{genuine} (\emph{artificial}) if $\Gf$ is orientation
preserving (respectively, reversing); it is called a \rm{solid},
\rm{dotted}, or \rm{bold} cut according to the structure of
$\Gamma\cap I_1$.
(The terms \rm{dotted} and \rm{bold} still apply
to cuts containing a \white-vertex.) A \emph{junction} is a genuine gluing 
 dessins along
isomorphic parts of two zigzags.

A dessin that is not equivalent to the result of gluing other
dessins is called \emph{indecomposable}.

\subsection{Unramified dessins.}
A dessin is called
\emph{unramified},
if all its \cross-vertices
are real.
In other words, unramified are the dessins corresponding
to maximally inflected curves.
In this subsection, we assume that~$\Gamma$
is an unramified dessin.

 Due to Theorem \cite[4.3.8]{DIZ}, the inner dotted edges of an unramified dessin of type I have \emph{types} $ 2 $ or $ 3 $.  Extend the notion of type of an inner dotted edge to the case of an unramified dessin of type II. Such an edge is of \emph{type} $ 3 $ or $ 2 $ if it can or cannot be adjacent to an inner \white-vertex (maybe after a \white-in).

The following theorem describes pillars of an unramified  dessin and its inner dotted edges. It is an immediate consequence of the definitions of unramified dessin and type of dotted edge.
\begin{theorem} \label{th.summary}
	Let $\Gamma$ be an unramified dessin of type~$\I$. Then
		\newcounter{N4}
	\begin{list}{\emph{(\arabic{N4})}}{\usecounter{N4}}
		\item
		the pillars of~$\Gamma$ are dotted \emph{(}ovals, zigzags\emph{)} or bold \emph{(}waves, jumps\emph{)};
		\item 
		a bold $ n $-pillar 
		is connected by $ n $ outgoing inner dotted edges of type~$ 2 $ with dotted		pillars, which are $ 0 $-pillars if $ \Gamma $ is bridge free; 
		\item two dotted pillars are interconnected by an inner dotted edge of type~$2$ or $3$,
		or a pair of edges of type~$3$ attached
		to an inner \white-vertex each;
		\item 
		if two dotted pillars 
		are interconnected by an inner dotted edge of type~$2$ then the edge is adjacent to a dotted bridge and the second inner edge adjacent to the bridge is of type~$2$. 
		\qed
	\end{list}
	\end{theorem} 

\section{Skeletons}\label{S.skeletons}
Here the notion of skeleton introduced in \cite{DIZ} for  maximally inflected trigonal curves of type~$\I$ is extended, in the case of curves over the rational base, to  maximally inflected curves of any type.

Unramified dessins  
can be reduced
to 
simpler objects, the so called skeletons,
which are obtained by disregarding all but dotted edges.
The principal result of this section
is Theorem~\ref{cor.Sk} describing
maximally inflected trigonal curves 
in terms of skeletons.

\subsection{Abstract skeletons}\label{a.skeletons}
Consider an embedded (finite) graph $\Sk \subset \bbase$ in a
disk $\bbase$. We do not exclude the possibility that
some of the vertices of~$\Sk$ belong to the boundary of~$\bbase$;
such vertices are called
\emph{real}, the other ones are called \emph{inner}. The set of edges at each real (respectively, inner)
vertex~$v$ of~$\Sk$ inherits from~$\bbase$ a pair of opposite linear
(respectively, cyclic) orders.
The \emph{immediate neighbors} of an edge~$e$ at~$v$ are the immediate
predecessor and successor of~$e$ with respect to (any) one of these orders.
A
\emph{first neighbor path} in~$\Sk$ is a sequence of
oriented
edges of~$\Sk$
such that each edge is followed by one of
its immediate neighbors.

Below, we consider graphs with 
edges of two kinds: directed and undirected. We call
such graphs \emph{partially directed}.
The directed and undirected parts (the unions of the corresponding edges and adjacent vertices) of a partially directed
graph~$\Sk$ are denoted by~$\Skdir$ and~$\Skud$, respectively. 

\begin{opr}\label{def.a.skeleton}
	Let~$\bbase$ be a disk.
	An \emph{abstract skeleton} is a partially directed
	embedded graph $\Sk \subset \bbase$,
	disjoint
	from the boundary $\partial\bbase$ except for some
	vertices,
	and satisfying the following conditions:
	\newcounter{N5}
	\begin{list}{(\arabic{N5})}{\usecounter{N5}}
		\item\label{Sk.1} 
		each  vertex is \emph{white} or \emph{black}, any white vertex is real, any inner black vertex is isolated, 
		any edge adjacent to a black vertex (named \emph{source}) is directed outgoing, any black/white isolated  vertex belongs to $ \Skdir/\Skud $;
		\item\label{Sk.2}
		any  immediate neighbor of an incoming 
		edge is an outgoing one; 		
		\item\label{Sk.3}
		$\Sk$ has no first neighbor cycles; 
		\item\label{Sk.4}
	the boundary $\partial\bbase$ has a vertex of~$\Sk$;
		\item\label{Sk.5}
		$b_1+3b=v+z$ for each
		component~$R$ of the closed
		cut  $\bbase_{\Sk}$ of $ \bbase $ along $\Sk $,
		where $b_1$ is the number of  black vertices with a single directed adjacent edge at  $\partial R$,  $ v $ is the number of (black)  vertices  with two adjacent outgoing edges at $\partial R$, 
		$b$ 
		is the number of   black isolated vertices in $ R $,
		$ z $ is the number of connected components of $\Skud\cap\partial R$;
	\end{list}

If additionally 
\newcounter{N6}
\begin{list}{(\arabic{N6})}{\usecounter{N6}}
	\addtocounter{N6}{\value{N5}}
	\item\label{Sk.8} 
	$ \Skdir\cap\Skud=\varnothing $; 
	\item\label{Sk.7} at each vertex, there are no 
	directed outgoing edges that are immediate neighbors; 
		\item\label{Sk.9} each white vertex of $\Skdir$ has odd valency and is a \emph{sink} which means that the number of its adjacent incoming  edges is one greater than the number of outgoing edges;
	\item\label{Sk.10} each black vertex  is real  and monovalent (thus, is a source);
	\item\label{Sk.11} 	each boundary component~$l$ 
	of~$\bbase$ is subject
	to the \emph{parity rule}: vertices of~$\Skdir$ and $\Skud $ alternate along~$l$,
\end{list} 
	\par\removelastskip
	then $ \Sk $ is a \emph{type~$\I$ skeleton} and the condition (\ref{Sk.5}) could be omitted due to (\ref{Sk.11}).
\end{opr}
\begin{za}
	\label{Sk1}
	The  abstract skeleton $ \Sk $   defined in \cite[5.1.1]{DIZ}  is a type~$\I$ skeleton 
The condition (1) of Definition \cite[5.1.1]{DIZ} is fulfilled by (\ref{Sk.2}) and (\ref{Sk.7}) of Definition \ref{def.a.skeleton}, the conditions (2) and (3) by (\ref{Sk.9}) and (\ref{Sk.10}), the condition (4) by (\ref{Sk.3}) and (\ref{Sk.10}), the condition (5) by (\ref{Sk.4}) and (\ref{Sk.11})  
of Definition \ref{def.a.skeleton}, the condition (6) due to $ \bbase $ is a disk.   
\end{za}

\subsection{Admissible orientations}

\begin{opr}\label{def.adm.or}
	Let~$\Sk \subset \bbase$ be an abstract skeleton. Its \emph{orientation} is the directions of edges of $\Skdir$ together with arbitrary directions of edges of $ \Skud $.
		An orientation of~$\Sk$ 
		is called \emph{admissible} if, at each vertex, no two incoming
		edges are immediate neighbors.
		An \emph{elementary inversion} of an admissible orientation
		is the
		reversal
		of the direction for one of the edges of
		$\Skud$ so that the result is again an admissible orientation.
\end{opr}

\begin{prop} \label{prop.adm.or1}
	Any abstract skeleton~$\Sk$ has an admissible orientation.
	Any two admissible orientations of~$\Sk$
	can be connected by a sequence of elementary inversions.
\end{prop} 

\begin{prop} \label{prop.adm.or2}
	Let $\Sk \subset \bbase$ be an abstract skeleton
	and $e_1$, $e_2$ two distinct edges of $ \Skud $.
	Then, out of the four \emph{states} \emph{(}orientations\emph{)} of the pair $e_1$, $e_2$,
	at least three extend to an admissible orientation.
\end{prop} 

\begin{proof} [Proof of Propositions~\ref{prop.adm.or1}
	and~\ref{prop.adm.or2}]
	In the case of type~$\I$ skeleton, the proof of the corresponding propositions \cite[5.2.2, 5.2.3]{DIZ}  uses first neighbor paths in $\Skud $. In the general case, due to the condition (\ref{Sk.2}) of  Definition \ref{def.a.skeleton}, any first neighbor path cannot return from $ \Skdir $ to $\Skud $. So it is sufficient to consider 
	such paths only in $ \Skud $ and the proof for type~$\I$ skeletons fits for the general case.
\end{proof} 

\subsection{Equivalence of abstract skeletons}\label{equivalence}
Two abstract skeletons
are called \emph{equivalent} if, after a homeomorphism
of underlying disks,
they can be connected by a finite sequence of isotopies
and the following \emph{elementary moves},
\emph{cf}.~\ref{ss.equivalence}:
\begin{itemize}
	\item[--] \emph{elementary modification}, see Figure~\ref{fig.Sk} (a);
	\item[--] \emph{creating} (\emph{destroying}) \emph{a bridge},
	see Figure~\ref{fig.Sk} (b);
	the vertex shown in the figure is white, 
	other edges of $\Sk$ adjacent to the vertex
	 may
	be present;
	\item[--] \emph{creating} (\emph{destroying}) \emph{an undirected edge},
	see Figure~\ref{fig.Sk} (c); the vertex shown in the figure is black real, the adjacent  edges are immediate neighbors, other adjacent directed outgoing edges  may	be present, the edge on the right side of the figure is undirected;
	\item[--] \emph{\black-in} and its inverse \emph{\black-out}, see
	Figure~\ref{fig.Sk} (d), (e); all vertices shown in the figures  
	are black, other  directed outgoing edges  adjacent to the real vertices	 may	be present in the figure (d).
\end{itemize}
\begin{figure}[tb]
	\begin{center}
		\includegraphics{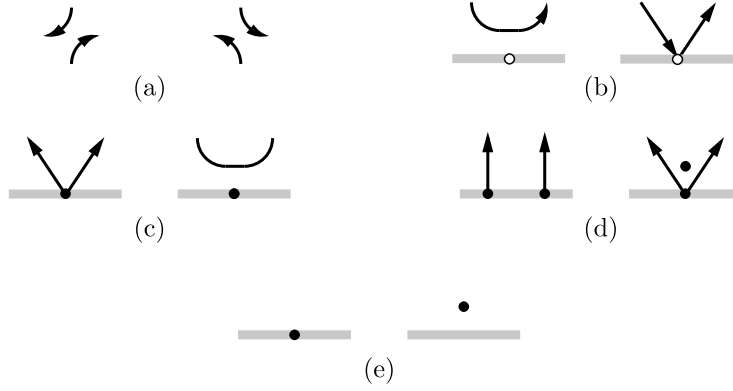}\\
	\end{center}
	\caption{Elementary moves of skeletons}\label{fig.Sk}
\end{figure}
(A move is valid only if the result is again
an abstract skeleton.)
Moves (a) and (b) must 
respect some admissible orientation
of the original skeleton
and take
it to an admissible orientation of the result.

An equivalence of two abstract skeletons in the same disk
and with the same set of vertices is called \emph{restricted}
if the homeomorphism is identical and
the isotopies above can be chosen identical on the vertices.

\subsection{Dotted skeletons}\label{s.skeletons}
Intuitively, a dotted skeleton is obtained from a
dessin $\Gamma$ by disregarding all but \rm{dotted} edges
and patching the latter through all \white-vertices.
According to  Theorem~\ref{th.summary}, each inner
dotted edge of type~$2$ retains a well defined orientation,
whereas an edge of type~$3$ may be broken by \white-vertex,
and for this reason, its orientation may not be defined.

\begin{opr}\label{def.skeleton}
	Let
	$\Gamma \subset \base$
	be an unramified dessin.
	The (\emph{dotted}) \emph{skeleton} of~$\Gamma$
	is a partially directed graph
	$\Sk=\Sk_\Gamma\subset\bbase$
	obtained from~$\Gamma$ as follows:
	\begin{itemize}
		\item[--] contract each
		pillar
		to a single
		point and declare this point a vertex of~$\Sk$, white for a maximal \rm{dotted} segment and black for a \rm{bold} one;
		\item[--] replace each inner \black-vertex of~$\Gamma$ by an  inner black vertex of~$\Sk$;
		\item[--]
		patch each inner \rm{dotted} edge
		through its
		\white-vertex, if there is one,
		and declare the result an edge of~$\Sk$;
		\item[--] let~$\Skdir$ ($\Skud$) be  the union of  black (respectively, white) isolated vertices and the closures of the images
		of the  edges of type~$2$ (respectively, $3$), 
		each edge of type~$2$ inheriting its
		orientation from~$\Gamma$.
	\end{itemize}
	Here, $\bbase$ is the
	disk
	obtained from~$\base$ by contracting
	each pillar to a single point.
\end{opr}

\begin{prop} \label{prop.Sk=Sk}
	The skeleton~$\Sk$ of a dessin~$\Gamma$
	as in Definition~\ref{def.skeleton}
	is an abstract skeleton in the sense of
	Definition~\ref{def.a.skeleton}.
\end{prop} 

\begin{proof} 
	Properties~(\ref{Sk.1}), (\ref{Sk.2}) of Definition \ref{def.a.skeleton}
	follow immediately from
	Theorem~\ref{th.summary}. Property~(\ref{Sk.4}) is fulfilled since all the \cross-vertices of $\Gamma$ are real and, hence, $\Gamma$ has pillars.
	Property~(\ref{Sk.5}) is fulfilled since the left-hand (right-hand) side of the equality is the number of outgoing  (resp., incoming) inner bold edges in $ R $. The proof of Property~(\ref{Sk.3}) given in \cite[5.4.3]{DIZ}  fits for a skeleton of any type.
\end{proof} 

\begin{prop} \label{Sk.extension}
	Any abstract skeleton $\Sk\subset\bbase$ is the
	skeleton
	of a certain dessin~$\Gamma$
	as in Definition~\ref{def.skeleton};
	any two such dessins can be connected by a sequence
	of isotopies and elementary moves,
	see~\ref{ss.equivalence},
	preserving the skeleton.
\end{prop} 

\begin{prop} \label{prop.Sk}
	Let $\Gamma_1, \Gamma_2 \subset \base$
	be two
	dessins as in Definition~\ref{def.skeleton};
	assume that~$\Gamma_1$ and~$\Gamma_2$
	have the same pillars.
	Then, $\Gamma_1$ and $\Gamma_2$
	are related by a restricted equivalence
	if and only if so are the corresponding
	skeletons~$\Sk_1$ and~$\Sk_2$.
\end{prop} 

Propositions~\ref{Sk.extension} and~\ref{prop.Sk}
are proved in Subsections~\ref{proof.prop} and~\ref{proof.prop1}, some points of proofs of the cooresponding Propositions \cite[5.4.4, 5.4.5]{DIZ} being used.
Here, we state the following immediate consequence.

\begin{theorem} \label{cor.Sk}
	There is a canonical
	bijection
	between the set of
	equivariant fiberwise deformation classes of maximally inflected
	real trigonal curves  and the set of equivalence
	classes
	of abstract skeletons.
	\qed
\end{theorem} 

\subsection{Proof of Proposition~\ref{Sk.extension}}\label{proof.prop}
The underlying disk~$\base$ containing~$\Gamma$
is the orientable blow-up $\beta:\base\rightarrow\bbase $ of~$\bbase$ at the real vertices of~$\Sk$:
each boundary vertex~$v$ is replaced with
the  segment of directions at~$v$. 
The  segments inserted are the
pillars
of~$\Gamma$.
Each source (real isolated black vertex) gives rise to a jump or a wave (resp., a wave without \white-vertices) and is decorated accordingly;
all other pillars consist of dotted edges (the \white-vertices
are to be inserted later, see~\ref{ss.white})
with \cross-vertices at the ends.
The proper transforms of the edges of~$\Sk$
are the inner dotted edges of~$\Gamma$.

\subsubsection{ }\label{ss.white}
The blow-up produces
a certain
\rm{dotted} subgraph $\Sk'\subset\base$.
Choose an admissible orientation of~$\Sk$, see
Proposition~\ref{prop.adm.or1}, regard it as an orientation of the
inner edges of~$\Sk'$, and insert a \white-vertex at the center of
each  dotted (bold) real segment connecting a pair of outgoing inner edges
and/or \cross-vertices (resp., \black-vertices).

\subsubsection{ }\label{ss.solid}
At the boundary $\partial D$, connect two neighboring  pillars  with a pair of real solid edges through a monochrome vertex if the pillars are of the same kind (dotted or bold),   and with a real solid edge otherwise. 
 
Let~$R$ be a region of the cut $\bbase_{\Sk}$, and $R'$ be the proper transform of $ R $, which is a region  of the cut $D_{\Sk'}$.  
 
The  vertices
of~$R$ define germs of inner solid and bold edges at the boundary $\partial R'$ and in the inner \black-vertices of $ R' $. 
The orientations of the germs and the valences of the \black- and the \white-vertices are determined by the definition of dessin. 
The arrangement of monochrome vertices indicated above shows that the number of germs of outgoing  (incoming) bold edges equals $b_1+3b$ (resp., $v+z $), the number of germs of outgoing  (incoming) solid  edges equals $v+z-b_r $ (resp., $b_1+2b_r+3b_i$) where $ b_r $ and $ b_i $ are the numbers of real and inner black vertices in $ R $.  Thus the numbers of outgoing and incoming germs of bold (solid) edges  are equal by (\ref{Sk.5}) of Definition~\ref{def.a.skeleton}. In particular, 
 the number
 of germs of outgoing   solid  edges 
is greater than the number of  inner 
black vertices  
  due to (\ref{Sk.5}) of Definition~\ref{def.a.skeleton}.  
  
  Connect each inner \black-vertex with a (real solid) monochrome vertex by a solid edge without crossing the obtained edges (clearly, it is possible). Consider  an oriented cycle $ \sigma $ along  $\partial R'$ with moving to each inner \black-vertex along the solid edge and back to $\partial R'$. It induces a cyclic order on the set of germs of solid and bold edges in $ R' $. Take a  pair of neighboring  incoming and outgoing bold germs. If there are  solid germs between them then the orientations of that bold  and  solid germs alternate along $\partial R'$ due to  the definition of dessin. Thus the solid germs form pairs of neighboring  incoming and outgoing  germs; convert each pair to a solid edge. Then convert  the initial pair of bold germs to a bold edge, cut off along it the obtained  regions of ~$\Gamma$ from  $ R' $, and reconstruct the cycle $ \sigma $ correspondingly. Repeat the process using the new cycle. 

This completes the construction of a dessin extending~$\Sk'$.

\subsubsection{ }\label{ss.uniqueness}
For the uniqueness, first observe that
a decoration of~$\Sk'$ with \white-vertices
is unique up to
isotopy and \white-ins/\white-outs along \rm{dotted} edges.
Indeed, assuming
all \white-vertices real,
each such decoration is obtained from a
certain admissible orientation, see~\ref{ss.white},
which is unique up to a sequence
of elementary inversions, see
Proposition~\ref{prop.adm.or1}, and an elementary
inversion
results in a \white-in
followed by a \white-out at the other end of the edge reversed.
Thus, the distribution of the \white-vertices can be assumed
`standard'.

For each region $R'$ of the cut $D_{\Sk'}$, make any dessin~$\Gamma$
extending $\Sk'$ `standard' in $ R' $ using elementary moves of dessins (see \ref{ss.equivalence}) in the following way.  
\subsubsection{ }\label{ss.scrap}
The region $ R' $ is a \emph{scrap} in  the sense of \cite[5.7]{DIK}. 
Each maximal oriented path of $\Skdir\cap\partial R$ (with its end named by  \emph{sink}) gives rise to 
a \emph{break} of the scrap, each connected component  of $\Skud\cap\partial R$ gives rise to a zigzag and each black vertex  with two adjacent outgoing edges at $\partial R$ gives rise to a pair of neighboring  \white-vertices   with two  real bold edges through a monochrome vertex at $\partial R'$. 
For a pair of \white-(\black-)vertices connected by  two bold edges, a \white-(\black-)in elementary moves converts the pair  
to a zigzag (resp., to an inner \black-vertex) and transforms $ R' $ to a scrap without  bold breaks and bold monochrome vertices. 
Let $ R' $ be such a scrap until the end of the proof.  

\subsubsection{ }\label{ss.source}
Let a white vertex $ \zeta $ be a  
neighbor of a source  at $ \partial R$. Then $ \beta^{-1}(\zeta) $ is a zigzag at $ \partial R'$ (see Figure \ref{2sources}). Performing, if necessary, a  bold monochrome modification and a solid bridge in the solid edge, for definiteness, not adjacent to the zigzag (Figure \ref{2sources} (a)), 
one obtains a region of $ \Gamma $ with two real solid edges, the first one adjacent  to a solid monochrome vertex, the second one to  a \cross-vertex. Hence, if the region is not triangular the edges are different and there is a solid artificial cut (the double line in the figure) that reduces  $ R' $ to a smaller scrap.
\begin{figure}[tb]
	\begin{center}
		\includegraphics{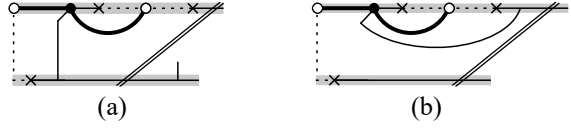}\\
	\end{center}
	\caption{Cut off a source and a \white-vertex}\label{2sources}
\end{figure}
\subsubsection{ }\label{ss.sink}
Let a white vertex $ \zeta $ be a 
neighbor of a sink  at $ \partial R$. Then $ \beta^{-1}(\zeta) $ is a zigzag at $ \partial R'$.
Performing, if necessary, a solid and a bold monochrome modifications, one obtains a region of $ \Gamma $ with two real solid edges, one of them adjacent  to a \black-vertex, another to  a \cross-vertex (see Figure \ref{sink_black} (a)). Hence, if the region is not triangular the edges are different and there is a solid artificial cut (the double line in the figure) that reduces  $ R' $ to a smaller scrap.  

\begin{figure}[tb]
	\begin{center}
		\includegraphics{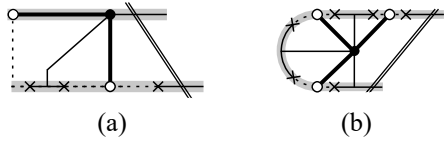}\\
	\end{center}
	\caption{Cut off (a) a sink and a \white-vertex, (b) a \black-vertex group}\label{sink_black}
\end{figure}

\subsubsection{ }\label{ss.black}
 Let   a  \white-vertex $ \omega $ be a  
neighbor of a  solid monochrome vertex $ \mu $ at $ \partial R'$, and they be connected with different inner \black-vertices $ \alpha_1 $ and $ \alpha_2 $ correspondingly. Then  in the region of $ \Gamma $ that contains  these four vertices, one can make a solid monochrome modification that connects $ \mu $ with $ \alpha_1 $ or a bold monochrome modification that connects $ \omega$ with $ \alpha_2 $. 
\subsubsection{ }\label{ss}
In view of (\ref{Sk.5}) of Definition \ref{def.a.skeleton}, applying \ref{ss.source} and \ref{ss.sink}, it is possible to reduce $ \Gamma\cap R' $ to a scrap  with a single dotted break and three triangular regions of $ \Gamma $ (a 'half' of type I cubic, see  Figure \ref{cubics} below) or to a dessin without real \black-vertices. In the latter case  $ b=3z $
due to (\ref{Sk.5}) of Definition \ref{def.a.skeleton}, and \white-vertices and solid monochrome ones  alternate along $ \partial R'$. Take a partition of the vertices into $ b $ groups of $ 6 $ consecutive vertices starting with a \white-vertex.  Due to \ref{ss.black} for any such partition one can have bold and solid edges of $ \Gamma $ that connect each \black-vertex with the vertices of a group. Note that, as in \ref{ss.source} and \ref{ss.sink}, for $ b\geq 1 $   there is a solid artificial cut (the double line in Figure \ref{sink_black} (b)) that reduces the scrap $ R' $ to a smaller scrap deleting such a group together with the corresponding \black-vertex.\qed

In contrast to the case of type I skeleton (see  \cite[5.4.4]{DIZ}), the uniqueness part of Proposition \ref{Sk.extension} is not true for type II skeletons of maximally inflected trigonal curves 
over a base curve
of positive genus.

\subsection{Proof
	of Proposition~\ref{prop.Sk}}\label{proof.prop1}
The `only if' part is obvious: an elementary move of a dessin
either leaves its skeleton intact or results in its
elementary modification; in the latter case, 
some admissible
orientation of  the skeleton  edges  is respected, see~\ref{ss.white}.

For the `if' part, consider the skeleton~$\Sk$
at the moment of
a transformation (a), (c) -- (e).
It can be regarded as the skeleton of an
admissible trichotomic graph 
and,
repeating the proof of
Proposition~\ref{Sk.extension}, one can see that $\Sk$ does
indeed
extend to a certain  
admissible trichotomic graph. The extension remains a valid
dessin~$\Gamma$ before the transformation as well. Hence, due to the
uniqueness given by Proposition~\ref{Sk.extension}, one can assume
that the original dessin is~$\Gamma$, and then the elementary move of
the skeleton is merely an elementary move of~$\Gamma$.

Destroying a bridge of a skeleton is the same as destroying
a bridge of the corresponding dessin, and the inverse operation
of creating a bridge extends to a dessin equivalent
to the original one due to the uniqueness given
by Proposition~\ref{Sk.extension}.
\qed

\section{A constructive description	of maximally inflected trigonal curves}\label{S.rational}
In this section,  a constructive description
of the real parts of maximally inflected 
trigonal curves is given.

\subsection{Blocks}\label{s.blocks}
\begin{opr}\label{def.block}
A \emph{type~$ \I $ cubic block} is an unramified  dessin  of degree~$ 3 $ of type~$\I$ (see Figure \ref{cubics} I). A \emph{type~$ \II $ cubic block} is an unramified  dessin of degree~$ 3 $ of type~II with an inner \black-vertex (see Figure \ref{cubics} II). Several cubic blocks glued artificially along solid edges is a (\emph{general}) \emph{block} (see corresponding possible artificial cuts in Figures  \ref{2sources} and \ref{sink_black}).
A type of a block is the type of the corresponding block curve.
\end{opr}
Due to \cite[5.6.7]{DIK}, type~I cubic block is unique, and the skeleton of an unramified  dessin of degree $ 3 $ of type~II has   a single real or inner black vertex, so    type~II cubic block is also unique. 

\begin{figure}[tb]
	\begin{center}
		\includegraphics{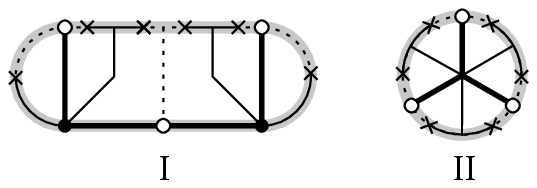}\\
	\end{center}
	\caption{}\label{cubics}
\end{figure}

Blocks are easily enumerated by the following statement.

\begin{prop} \label{block.existence}
		Let~$d\ge1$ be an integer, and let $O,J\subset S^1=\partial\base$ be
	two disjoint sets of size~$d$ each. Then, there is a unique, up to
	restricted equivalence, type I
	block $\Gamma\subset\base$ of degree~$3d$
	with an oval about each
	point of~$O$, a jump at each point of~$J$, and a zigzag between
	any two points of $O\cup J$ \emph{(}and no other pillars\emph{)}.
	
	A block 
	of degree~$3d$ of any type with $ c $ jumps, $ c $ ovals, $ b $ inner \black-vertices and $ z $ zigzags corresponds to an abstract skeleton in a disk  with $ c $ directed disjoint
	chords, $ b $  inner \emph{(}black\emph{)} vertices and $ z $  real isolated white vertices, that satisfy the following conditions: 
	\begin{enumerate}
		\item \label{11.1} 
	$b+c=d$, $z+c=3d$; 
	\item\label{11.2} 
	for each	component $ R_i $ of the closed cut of the disk along the chords,  $z_i=c_i+3b_i$ where $c_i$, $b_i$ and	$z_i$ are the numbers of the  chords, the black inner vertices
	and the  real isolated vertices of the component.
\end{enumerate}
\end{prop} 

\begin{proof} The case of type I block is proved in \cite[6.3.2]{DIZ}.

Let us prove the second part of the Proposition. Clearly, a block is an unramified dessin. From the proof of Proposition~\ref{Sk.extension} it follows that an unramified dessin is a block if and only if, for its skeleton $ \Sk $, the directed part $ \Skdir $ is a collection of inner black vertices and disjoint chords  in the disk~$\bbase$ connecting sources and sinks,  the undirected one $ \Skud $ is a collection of 
real white isolated vertices (zigzags). The conditions (\ref{11.1}) and (\ref{11.2}) follow from the dessin definition and  the abstract skeleton definition correspondingly.
\end{proof} 
\begin{za}
	\label{cor.block}
Whereas the first part of Proposition \ref{block.existence} enumerates the equivalence classes of type I blocks directly, the second part gives, due to Theorem \ref{cor.Sk}, a bijection between the equivalence classes of  blocks and the equivalence classes of abstract skeletons used in the Proposition. For example, there are exactly two type I blocks on $ \Sigma_2 $ with the skeletons shown in Figure \ref{blocks_on_s_2}.
	
\end{za}
\begin{figure}[tb]
	\begin{center}
		\includegraphics{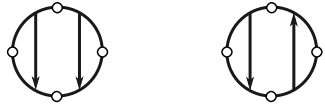}\\
	\end{center}
	\caption{Type I blocks on $\Sigma_2$}\label{blocks_on_s_2}
\end{figure}
\subsection{Gluing}
	\label{gluing}
	Given an unramified dessin $ \Gamma $, make all \white-out to obtain all \white-vertices real. This corresponds to choosing an admissible orientation of the skeleton of $ \Gamma $.
	Destroying all dotted bridges in $ \Gamma $ gives its skeleton consisting of isolated vertices and \emph{star-like} components. The center of a star-like component of $ \Skdir $ ($ \Skud $) is a  black (resp., white) real vertex, the end of each ray is a white vertex corresponding in $ \Gamma $ to an oval  free of \white-vertices. 
	Such a dessin is called \emph{star-like}. 
	
	Given  	star-like dessins $\Gamma_1 $, $ \Gamma_2 $, 
	glue them genuinely along a diffeomorphism  $ \varphi: l_1\rightarrow l_2 $ where  $ l_i\subset \Gamma_i$, $i=1,2$ are bold or dotted segments,  and destroy a possible bridge to obtain a star-like  dessin  $ \Gamma $. If $ l_1, l_2 $ are parts of zigzags then the gluing is a junction. Let $ l_i $ be a part of a bold or dotted edge of a pillar $ p_i \subset\Gamma_i$. Let the orientation of $ l_i $ be thought of  vertically down, and $ a_i $ and $ b_i $ be the numbers of \white-vertices (dotted monochrome vertices) of the bold (resp., dotted) pillar $ p_i $ above and below $ l_i$. Then in $ \Gamma $, after the gluing, $ p_1 $ and $ p_2 $ are converted to a pair of bold (dotted) pillars   
	with the numbers of \white-vertices (resp., dotted monochrome vertices) equal $ a_1+a_2 $ and $ b_1+b_2 $. It is clear how the corresponding  gluing of skeletons of $\Gamma_1 $, $ \Gamma_2 $ is described.
\subsection{Real parts of maximally	inflected curves}
\label{CR}
Proposition~\ref{block.existence} and Section \ref{gluing} 
provide a complete description of the real part of a maximally
inflected 
trigonal curve, 
\emph{i.e.}, a description of the topology of  $C_{\R}\cup s_0$ where $ s_0\subset\Sigma_d $ is the zero section.
Realizable is a real part obtained as follows: \begin{enumerate}
	\item For a curve of type I, start with a
disjoint union of a number of type I  blocks, see
Proposition~\ref{block.existence}, and perform a sequence of
junctions converting the disjoint union of disks to a single
disk.
\item For a curve of type II, start with a
disjoint union of a number of   blocks, see
Proposition~\ref{block.existence}, and glue the blocks
genuinely along parts of bold and/or dotted edges  converting the disjoint union of disks to a single
disk.
\end{enumerate}

\begin{za}\label{remark4}
	As it is mentioned in \cite[6.3.3]{DIZ}, 
	a decomposition	of an unramified	dessin 	into  
	blocks is far	from unique even in the case of type~I curves.
\end{za}

\section{Rigid isotopies and week equivalence}\label{s.rigid}
A \emph{rigid isotopy}
of nonsingular real trigonal curves is an isotopy in the class of
nonsingular real algebraic
curves
in a fixed real ruled surface.
During the isotopy the curves are not necessarily almost-generic. 

Intuitively, the new notion differs from the deformation equivalence
by an extra pair of mutually inverse operations:
straightening/creating a zigzag, the former consisting in
bringing the two vertical tangents
bounding
a zigzag
together to a single vertical flex and pulling them apart
to the imaginary domain. On the level of dessins,
these operations are shown in Figure~\ref{fig.zigzag}.

\begin{opr}
	Two dessins are called \emph{weakly equivalent} if they
	are related by a sequence of isotopies,
	elementary moves (see~\ref{ss.equivalence}),
	and the operations
	of \emph{straightening/\penalty0creating a zigzag}
	consisting in replacing one of the fragments shown in
	Figure~\ref{fig.zigzag} with the other one.
\end{opr}

\begin{figure}[tb]
	\begin{center}
		\includegraphics{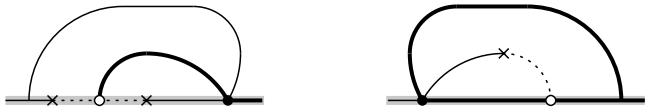}\\
	\end{center}
	\caption{Straightening/creating a zigzag}\label{fig.zigzag}
\end{figure}
The following statement is easily deduced from~\cite{DIK},
\emph{cf}.~Proposition~\ref{equiv.curves}.

\begin{prop} \label{equiv.zigzag}
	Two generic real trigonal curves are rigidly isotopic
	if and only if their dessins are weakly equivalent.
	\qed
\end{prop} 

As  mentioned in \cite{DIZ}, the following theorem 
can be deduced, \emph{e.g.}, from
Propositions~5.5.3 and~5.6.4 in~\cite{DIK},
see also~\cite{Z}.
\begin{theorem} \label{max.inflected}
	Any non-hyperbolic nonsingular
	real trigonal curve
	on a Hirzebruch surface is
	rigidly isotopic to a maximally inflected one. \qed
\end{theorem} 
Due to \cite[6.6.2]{DIK}, if trigonal curves are $M$-curves then their weak equivalence classes coincide with the classes of deformation equivalence. Is it true for maximally inflected curves? This   question was put in Remark \cite[A.3.2]{DIZ} in the case of curves of type I. Proposition \ref{rig.vs.def} below answers the question and, combined with Subsection \ref{CR} and Theorem \ref{max.inflected}, allows to describe the real parts of nonsingular real trigonal curves up to weak equivalence. 
\begin{lem}
	\label{jzo}
A weak equivalence can reverse the order of the triple $ (J,Z,O) $ where $ J $, $ Z $ and $ O $ are neighboring jump, zigzag  and oval \emph{(}cf. \emph{\cite[6.6.2]{DIK})}.	     
\end{lem}
\begin{proof}
See Figure \ref{lantern} with a fragment of  an intermediate dessin appeared during a weak eqiuvalence  after straightening the zigzag,  a \white- and a \black-in.	
\end{proof}
\begin{figure}[tb]
	\begin{center}
		\includegraphics{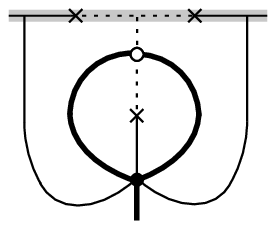}\\
	\end{center}
	\caption{Fragment of  an intermediate dessin}\label{lantern}
\end{figure}
\begin{prop} \label{rig.vs.def}
For any~$d\ge1$, there is a unique, up to
weak equivalence, type I
block $\Gamma\subset\base$ of degree~$3d$.
\end{prop} 
\begin{proof}
In the notation of Proposition \ref{block.existence} let there be a sequence of successive points $ o_0,o_1,j_1,o_2,j_2,\ldots,o_n,j_n,j_{n+1} $ at $\partial D $ where $ o_i\in O $, $ j_m\in J $. By Lemma \ref{jzo}, one can make $ n $ transpositions of neighboring points to get the sequence  $ o_0,j_1,o_1,j_2,o_2,\ldots,j_n,o_n,j_{n+1} $. Thus, by induction it is possible to obtain the points of $ O $ and $ J $ being alternate along $\partial D $.
\end{proof}

The author is grateful to the referee for indicating some inaccuracies and ambiguities in  the first version of the paper.

\end{document}